\documentclass{amsart}
\usepackage{amssymb}
\usepackage[alphabetic]{amsrefs}
\usepackage[latin1]{inputenc}

\newcommand\K{{\mathcal K}}

\newcommand\N{{\mathbb N}}

\newcommand\V{{\mathcal V}}

\newcommand\Z{{\mathbb Z}}
\newcommand\Zh{{\widehat\Z}}

\newcommand\at[1]{{@#1}}
\newcommand\prist[1]{_{\#{#1}}}
\newcommand\aut{\operatorname{Aut}}

\newcommand\core{\operatorname{Core}}

\newcommand\gb{\overline{\Gamma}}
\newcommand\gh{\widehat{\Gamma}}
\newcommand\gt{\widetilde{\Gamma}}
\newcommand\id{\operatorname{id}}

\newcommand\rist{\operatorname{Rist}}
\renewcommand\setminus{\smallsetminus}

\newcommand\stab{\operatorname{Stab}}

\newcommand\sym{\operatorname{Sym}}

\newtheorem{thm}{Theorem}[section]
\newtheorem{prop}[thm]{Proposition}
\newtheorem{cor}[thm]{Corollary}
\newtheorem{lem}[thm]{Lemma}

\theoremstyle{remark}
\newtheorem{rmk}[thm]{Remark}
\theoremstyle{definition}

\newtheorem{quest}{Question}
\newtheorem{prob}[thm]{Problem}

\newtheorem{mainthm}{Theorem}
\newtheorem{maincor}[mainthm]{Corollary}

\newtheoremstyle{Quotethm}{\topsep}{\topsep}%
     {\itshape}
     {}
     {\bfseries}
     {.}
     { }
     {\thmname{#1}\thmnote{ #3}}
\theoremstyle{Quotethm}
\newtheorem*{quotethm}{Theorem}
\newtheorem*{quotecor}{Corollary}

\numberwithin{equation}{section}

\font\cyreight=wncyr8

\begin{document}
\title{The congruence subgroup problem for branch groups}
\author{Laurent Bartholdi}
\author{Olivier Siegenthaler}
\author{Pavel Zalesskii}
\thanks{This project was initiated when the third author visited the
  first two at EPFL, whose financial support is gratefully
  acknowledged. Pavel Zalesskii is partly supported by CNPq.}
\begin{abstract}
  We state and study the congruence subgroup problem for groups acting
  on rooted tree, and for branch groups in particular.  The problem is
  reduced to the computation of the congruence kernel, which we split
  into two parts: the branch kernel and the rigid kernel.  In the case
  of regular branch groups, we prove that the first one is Abelian
  while the second has finite exponent.  We also establish some
  rigidity results concerning these kernels.

  We work out explicitly known and new examples of non-trivial
  congruence kernels, describing in each case the group action.  The
  Hanoi tower group receives particular attention due to its
  surprisingly rich behaviour.
\end{abstract}
\maketitle
\setcounter{tocdepth}{2}

\section{Introduction}
One often encounters infinite mathematical objects that are
arbitrarily well approximated by finite (quotient) objects. The
question then arises as to which properties of the infinite object one
can read in the finite objects.  As a classical
example~\cite{selmer:diophantine}, the equation $3x^3+4y^3+5z^3=0$
admits non-trivial solutions in $\Z/n\Z$ for all $n$, but no
non-trivial solution in $\Z$.

As another example, consider the integers $\Z$ approximated by the
finite rings $\Z/2^n\Z$. The integer $3$ is invertible in each of
these finite quotients, but not in $\Z$. More pedantically, the ideal
$3\Z$ has index $3$ in $\Z$ but its closure in $\varprojlim\Z/2^n\Z$ has
index $1$.

We concentrate on group theory. We have a group $G$ acting on a set
$X$; suppose this set admits $G$-equivariant finite quotients $X_n$,
on which $G$ induces permutation groups $G_n$. Saying that $G$'s
action on $X$ is arbitrarily well approximated by actions on the $X_n$
amounts to saying that $G$ embeds in the profinite group $\overline
G=\varprojlim G_n$. On the other hand, $G$ also embeds in its profinite
completion $\widehat G$, which maps naturally onto $\overline G$. The
\emph{congruence problem} asks to compare these profinite groups, or
more quantitatively to describe the kernel of the map $\widehat
G\to\overline G$, the \emph{congruence kernel}.

Let $G$ be a group acting faithfully on a locally finite rooted tree
$T$. Then $G$ is residually finite since the intersection $\bigcap_n
\stab_G(n)$ of the pointwise stabilizers of the levels $T_n$ is
trivial. It follows that $G$ is Hausdorff with respect to two
topologies: the profinite topology and the topology defined by level
stabilizers $\stab_G(n)$. We shall call the latter \emph{congruence
  topology}. Therefore, as in the classical congruence subgroup
problem~\cite{bass-l-s:congruence} we can ask whether these two
topologies coincide.  We can follow this analogy further, completing
$G$ with respect to these two topologies, and define a
\emph{congruence kernel} $C$ as the kernel of the natural epimorphism
of the completions $\widehat G\longrightarrow \overline G$ (it is
worth mentioning that $\overline G$ coincides with the closure of $G$
in the profinite group $\aut(T)$). Then the congruence subgroup
problem can be formulated as computation of the congruence kernel.

Denote by $\rist_G(v)$ and call the \emph{rigid stabilizer} of $v$ the
subgroup of $G$ consisting of the elements that move only vertices
which are in the subtree rooted at $v$.  Let $\rist_G(n)$ be the group
generated by $\rist_G(v)$ for all $v$ at level $n$. We say that $G$ is
a \emph{branch group} if it admits an action on a rooted tree $T$ such
that $\rist_G(n)$ is of finite index for every $n$. The class of
branch groups has been studied intensively during the last two
decades~\cite{bartholdi-g-s:bg}. We also say the group $G$ is
\emph{level-transitive} if it acts transitively on each level
$T_n$. Note that we make that assumption in Theorems A-F, but we do
not require it as part of the definition of a branch group, so as to
cover potentially interesting examples such as those
in~\cite{nekrashevych:cantor}.

If $T$ is a regular rooted tree, meaning the number of descendants of
every vertex is constant, we say $G\le\aut(T)$ is \emph{self-similar}
if for every $g\in G$ and $v\in T$ the composition $T\to T_v\to
T_{v^g}\to T$ defines an element of $G$; here $T_v$ denotes the
subtree below $v$.

For a branch group, rigid stabilizers of levels define another
profinite topology that we shall call \emph{branch topology}. This
topology is stronger than the congruence and weaker than the
profinite topology. Denote by $\widetilde G$ the completion of $G$
with respect to the branch topology. Then for the branch group $G$
we have the following kernels:

\medskip
Congruence kernel: $\ker(\widehat G\to\overline G)$;

Branch kernel: $\ker(\widehat G\to\widetilde G)$;

Rigid kernel: $\ker(\widetilde G\to\overline G)$.

\bigskip

In this paper we are concerned with determining the structure of these
kernels. However the preliminary question of great importance is

\begin{quest}
Do any of the kernels depend on the branch representation of
$\Gamma$?
\end{quest}

\noindent In other words, are these kernels invariants of the group
$G$, or only of its action on a tree?

\medskip The situation is quite clear in the case of self-similar,
regular branch groups.  A group $G$ acting faithfully on a rooted tree
is \emph{regular branch} if there is a subgroup $K$ in $G$ such that
$K$ contains a copy $v*K$ of $K$ acting on the subtree rooted at $v$
for every vertex $v$ of the tree, and the subgroup generated by the
$v*K$ for all $v$ of given length has finite index in $G$.  For a
regular branch group the congruence kernel does not depend on the
representation.  We shall prove the following

\newcommand\regbranchindep{%
  Let $\Gamma$ be a group and consider two level-transitive,
  self-similar, regular branch representations $\Gamma\to\aut T_i$
  with corresponding closures $\gb_i$. The congruence kernel
  $\K_1=\ker(\gh\to\gb_1)$ is equal to the congruence kernel
  $\K_2=\ker(\gh\to\gb_2)$.  }
\begin{mainthm}\label{thm:regular branch -> independence}
\regbranchindep
\end{mainthm}

\noindent This theorem uses the following structure result:

\begin{mainthm}\label{thm:structure of congruence kernel}
  Let $\Gamma$ be a level-transitive, self-similar, regular branch
  group.  Then the branch kernel is Abelian and the rigid kernel has
  finite exponent.
\end{mainthm}

Summarizing the information for the congruence kernel we can state
the following

\begin{mainthm}
  Let $\Gamma$ be a level-transitive, self-similar, regular branch
  group. Then the congruence kernel $\ker(\gh\to\gb)$ does not depend
  on a branch representation and is Abelian by finite exponent.
\end{mainthm}

For general branch groups (i.e. not \emph{regular} branch) the
situation is possibly more complicated.  We have the following partial
result concerning the independence of the representation:

\newcommand\indepbranchkernel{%
  Let $\Gamma$ be a level-transitive branch group and consider two
  level-transitive branch representations $\rho_i:\Gamma\to\aut T_i$
  of $\Gamma$ with $i=1,2$. Write $\gt_i$ and $\gb_i$ for the
  corresponding branch and congruence completion and denote the branch
  kernels by $\K_i=\ker(\gh\to\gt_i)$.

  Then $\K_i$ is in the congruence kernel $\ker(\gh\to\gb_j)$ for all
  $i,j\in\{1,2\}$.  }
\begin{mainthm}\label{thm:independence branch kernel}
\indepbranchkernel
\end{mainthm}

That is, the branch kernel is contained in the congruence kernel of
any other representation.  This theorem is heavily based on the
observation that the branch kernel is Abelian. Moreover, it can be
considered as a profinite $\gt$-module that is described by the
following

\newcommand\descrbranchkernel{%
  Let $\Gamma\leq\aut T$ be a level-transitive branch group.  There is
  an Abelian group $A$ such that the branch kernel
  $\K=\ker(\gh\to\gt)$ is isomorphic to
  $\widehat\Z[[\gt]]\widehat\otimes_{\widehat\Z[[\gt_w]]}A$ qua
  profinite $A[[\gt]]$-module.  Here $\gt_w$ is the stabilizer in
  $\gt$ of a point $w\in\partial T$ of the boundary of the tree, where
  the action of $\gt$ on $T$ is given by the natural map
  $\gt\to\gb\leq\aut T$.

Moreover, $A$ can be explicitly computed as
\begin{align*}
\varprojlim_{\substack{v<w\\k\geq0}}\rist_\Gamma(v)/(\rist_\Gamma(v)'\rist_\Gamma(v)^k).
\end{align*}
}
\begin{mainthm}\label{thm:description branch kernel}
\descrbranchkernel
\end{mainthm}

\newcommand\branchkernelfinexpo{%
  Let $\Gamma$ be a just-infinite, level-transitive, self-similar,
  regular branch group. The branch kernel $\K=\ker(\gh\to\gt)$ has
  finite exponent.  }
\begin{maincor}\label{cor:branch kernel finite exponent}
\branchkernelfinexpo
\end{maincor}

The first main examples, the Grigorchuk and Gupta-Sidki groups, were
shown in~\cite{bartholdi-g:parabolic} to have a trivial congruence
kernel. This information was vital in obtaining a description of the
top of the subgroup lattice of these groups -- the lattice was
computed in a finite quotient, and validated by a quantitative version
of the congruence property. Pervova constructed
in~\cite{pervova:profinitecompletions} the first examples of groups
generated by automata without the congruence property. We revisit her
examples, proving that these groups have trivial rigid kernel, and
giving additional information on the branch kernel (as in particular
its structure as a $\Gamma$-module).  We compute the kernels for the
Hanoi tower group, which is the first example where both the rigid and
the branch kernel are non-trivial.

\subsection{Basic definitions and notation}
We introduce the necessary vocabulary describing groups acting on
rooted trees. The main notions are summarized in Table~\ref{table:1}.

\begin{table}[h]
\hspace{-1cm}\begin{tabular}[c]{|c|ll|p{85mm}|}
\hline
symbol&\multicolumn{2}{l|}{name}&definition\\
\hline
\rule{0mm}{2.5ex}
$T$ & \multicolumn{2}{l|}{rooted tree} & \\
$X^*$ & \multicolumn{2}{l|}{regular, rooted tree} & \\
$X^n$ & \multicolumn{2}{l|}{$n$-th level of $X^*$} & \\
$g\at{w}$ & \multicolumn{2}{l|}{state of $g$ at $w$} & projection of $g$ on the $w$-th coordinate in $\aut X^*\to\aut X^*\wr\aut X^n$\\
$w*g$ & \multicolumn{2}{l|}{$w$-translate of $g$} & element of $\rist(w)$ with $(w*g)\at{w}=g$\\
$G\prist v$ & & & those $g\in G$ with $v*g\in G$\\
$K$ & \multicolumn{2}{l|}{branching subgroup} & subgroup with $v*K\leq K$ for all $v\in X^*$\\
$G_\#$ & \multicolumn{2}{l|}{ } & maximal branching subgroup of $G$\\
$X^n*K$ & & & product of $v*K$ for all $v\in X^n$\\
\hline
\rule{0mm}{2.5ex}
$\stab_G(w)$ & \multicolumn{2}{l|}{stabilizer of $w$} & subgroup of $G$ consisting of the elements which fix $w$\\
$\stab_G(n)$ & \multicolumn{2}{l|}{$n$-th level stabilizer} & intersection of $\stab_G(w)$ with $w$ ranging over $X^n$\\
$\rist_G(w)$ & \multicolumn{2}{l|}{rigid stabilizer of $w$} & subgroup of $G$ consisting of the elements moving only descendants of $w$\\
$\rist_G(n)$ & \multicolumn{2}{l|}{$n$-th rigid stabilizer} & group generated by $\rist_G(w)$ with $w$ ranging over $X^n$\\
\hline
\multicolumn{2}{|l|}{\rule{0mm}{2.5ex}level-transitive} & \multicolumn{2}{l|}{the action of $G$ is transitive on each level}\\
\multicolumn{2}{|l|}{self-similar} & \multicolumn{2}{l|}{the image of $G\to\aut X^*\wr\sym(X)$ is contained in $G\wr\sym(X)$}\\
\multicolumn{2}{|l|}{recurrent} & \multicolumn{2}{l|}{$\stab_G(v)\at{v}=G$ for all $v\in X^*$}\\
\multicolumn{2}{|l|}{weakly branch} & \multicolumn{2}{l|}{$\rist_G(v)\neq1$ for all $v\in X^*$}\\
\multicolumn{2}{|l|}{branch} & \multicolumn{2}{l|}{$\rist_G(n)$ has finite index in $G$ for all $n\geq0$}\\
\multicolumn{2}{|l|}{weakly regular branch} & \multicolumn{2}{p{100mm}|}{$G$ contains a non-trivial branching subgroup}\\
\multicolumn{2}{|l|}{regular branch} & \multicolumn{2}{p{100mm}|}{weakly regular branch, and $X^n*K$ has finite index in $G$ for all $n\ge0$}\\
\hline
\end{tabular}
\caption{Symbols, subgroups, and main properties of groups acting on rooted trees.}
\label{table:1}
\end{table}

\subsubsection*{Rooted trees}
A \emph{rooted tree} is a directed graph which is connected, and such
that there is a unique vertex with no incoming edge (the \emph{root})
and all the other vertices have exactly one incoming edge.  The
\emph{level} of a vertex $v$ is the combinatorial distance between the
root and $v$.  A rooted tree is \emph{locally finite} if every vertex
has finite valency.  The \emph{boundary} of a rooted tree $T$ is the
set $\partial T$ consisting of infinite directed paths starting from
the root. It is a Cantor set if $T$ is locally finite but has infinite
boundary.

A locally finite rooted tree $T$ is \emph{spherically homogeneous} if
there are positive integers $d_i$ such that every vertex of level $i$
has exactly $d_i$ outgoing edges.  The sequence $d_i$ is the
\emph{signature} of $T$.  A \emph{regular} rooted tree is a
spherically homogeneous tree whose signature is a constant sequence.
A \emph{$d$-ary} rooted tree is a regular rooted tree with signature
$d_i=d$.  We shall identify the vertices of a $d$-ary rooted tree with
the elements of the free monoid $X^*$ with basis $X$, where $X$ is a
set with $d$ elements.  We shall use the notation $X^*$ for regular
rooted trees, and use $T$ to denote arbitrary spherically homogeneous
rooted trees.  In this notation, the root corresponds to the identity
of $X^*$, and we shall let $X^n$ designate the set of vertices of
level $n$.  We shall often refer to \emph{words of length $n$} to
designate the elements of $X^n$, i.e. the vertices of level $n$.  The
boundary of $X^*$ is $X^\omega$, the Cantor set of (right) infinite
sequences over the alphabet $X$.

Consider two vertices $u,v$ of a rooted tree $T$.  We say $u$ is a
\emph{prefix} of $v$ (or $v$ is a \emph{descendant} of $u$) if there
is a directed path going from $u$ to $v$.  In such a situation we
write $u\leq v$.  The set of descendants of $u$ form the \emph{subtree
  rooted at $u$}.

\subsubsection*{Automorphisms of rooted trees}
An automorphism of a rooted tree is a directed graph automorphism,
i.e. a permutation of the vertices which preserves the edges and their
orientation.  It is easy to see that an automorphism preserves the
level of every vertex.

Let $T$ be a rooted tree and consider a vertex $v$ of $T$.  The
\emph{rigid stabilizer} of $v$ is the subgroup of $\aut T$ consisting
of the elements that move only vertices which are in the subtree
rooted at $v$.  We denote this subgroup by $\rist(v)$.  If $T$ is
spherically homogeneous then $\rist(u)$ is isomorphic to $\rist(v)$
when $u$ and $v$ are on the same level.

Let $v$ be a vertex of $T$.  The \emph{stabilizer} of $v$ is the group
$\stab(v)$ of elements of $\aut T$ which fix $v$.  The \emph{$n$-th
  level stabilizer}, written $\stab(n)$, is the subgroup of $\aut T$
consisting of the elements that fix all the vertices of the $n$-th
level.  It is a normal subgroup of $\aut T$, isomorphic to
$\prod_v\rist(v)$ where $v$ ranges over the vertices of level $n$.
The group $\aut T$ is the projective limit of the projective system
$\aut T/\stab(n)$.  If $T$ is locally finite then the groups $\aut
T/\stab(n)$ are finite and $\aut T$ is profinite.

If $T\simeq X^*$ is a regular rooted tree, we shall often use the
notation $\aut X^*$ for the automorphism group of $X^*$, and $\aut
X^n$ for the finite quotient $\aut X^*/\stab(n)$.

\subsubsection*{The automorphism group of a rooted tree}
Let $T_0$ be a spherically homogeneous rooted tree with signature
$(d_i)$.  Let $T_n$ be a rooted tree isomorphic to a subtree of $T_0$
rooted at a vertex of level $n$ ($T_n$ is independent of the choice of
the vertex because $T_0$ is spherically homogeneous).  There is a
natural isomorphism $\aut T_0\simeq\aut T_1\wr\sym(d_0)=(\aut
T_1)^{d_0}\rtimes\sym(d_0)$, and more generally
\begin{align*}
\aut T_0&\xrightarrow{\sim}\aut T_n\wr W_n,\\
g&\mapsto((g\at{w})_{w\in L_n},\sigma),
\end{align*}
where $W_n=(\cdots(\sym(d_{n-1})\wr\sym(d_{n-2}))\cdots)\wr\sym(d_0)$
is a permutation group acting on $D_n=\prod_{i=0}^{n-1}d_i$ points,
and $L_n$ is the set of vertices of level $n$.  Note that the above
defines $g\at{w}$ for any vertex $w$ of the tree as the image of $g$
under the composition $\aut T_0\to\aut T_n\wr W_n\to(\aut
T_n)^{D_n}\to\aut T_n$, where the first and the last arrows are group
homomorphisms, whereas the middle one is just a set map.  We define
$w*g=(1,\dots,1,g,1,\dots,1)$ where $g$ appears in position $w$.

If $X^*$ is a $d$-ary rooted tree, then $\aut X^*$ is isomorphic to
the iterated wreath product $\varprojlim \wr^n\sym(d)$.  In this case
one can view $g\mapsto g\at{w}$ and $g\mapsto w*g$ as maps from $\aut
X^*$ into itself.  In this situation, a subgroup $G\leq\aut X^*$ is
called \emph{self-similar} if $G\at{w}$ is contained in $G$ for all
$w\in X^*$.

\subsubsection*{Sylvan representations}
Let $G$ be a group.  A \emph{sylvan representation} of $G$ is a group
homomorphism $\rho:G\to\aut T$ where $T$ is a locally finite rooted
tree.  The representation $\rho$ is \emph{faithful} if $\rho$ is
injective.  The representation $\rho$ is \emph{level-transitive} if
the action of $\rho(G)$ is transitive on each level of $T$.  Clearly
$T$ must be spherically homogeneous for this condition to be
satisfied.

Let $\rho$ be a faithful representation of $G$, and identify $G$ with
its image in $\aut T$. Define $\rist_G(v)=\rist(v)\cap G$ and $G\prist
v=\rist_G(v)\at v$. Let $\rist_G(n)$ be the group generated by
$\rist_G(v)$ for all $v$ of level $n$. The group $\rist_G(n)$ is a
normal subgroup of $G$. The representation $\rho$ is \emph{weakly branch} if
$\rist_G(v)$ is non-trivial for all $v\in T$. The representation $\rho$ is \emph{branch} if
$\rist_G(n)$ has finite index in $G$ for all $n\geq0$. A group $G$ is
\emph{(weakly) branch} if there exists a faithful (weakly) branch representation of $G$.
The representation $\rho$ is \emph{weakly regular branch} if there is a
subgroup $K\neq1$ of $G$ such that $v*K\leq K$ for all $v\in X^*$.
The representation $\rho$ is \emph{regular branch} if moreover $\prod_{v\in X^n}v*K$ has finite index in $G$ for all $n\ge0$.  A
group $G$ is \emph{(weakly) regular branch} if there exists a faithful (weakly) regular
branch representation of $G$.  In this situation the subgroup $K$
shall be called a \emph{branching subgroup} of $G$, and we shall
denote by $X^n*K$ the product of all $v*K$ with $v\in X^n$. Note
that, contrary to previous literature, we do not require
level-transitivity as part of the definition.

\begin{lem}\label{lem:(regular) branch}
  Let $G$ be a (weakly) regular branch group.  Then $G$ is
  (weakly) branch.
\end{lem}
\begin{proof}
  Identify $G$ with its image in $\aut X^*$. Now $\rist_G(v)$
  obviously contains $v*K$, and thus a weakly regular branch group is regular branch.
  Further, $\rist_G(n)$ contains $X^n*K$, which has finite index in $G$ if $G$ is regular branch.
\end{proof}

\begin{lem}
  Let $G$ be a regular branch group with faithful sylvan
  representation $\rho$. Then there is a unique maximal subgroup
  \[G_\#=\bigcap_{v\in X^*}G\prist v\] of $G$ that is a branching
  subgroup with respect to $\rho$.
\end{lem}
\begin{proof}
  Once again, identify $G$ with its image under $\rho$. Let $K$ be a
  branching subgroup. Then $v*K\le K\le G$ for all $v\in X^*$, so
  $v*K\le\rist_G(v)=v*G\prist v$. It follows that $K\le G\prist v$ for
  all $v\in X^*$, so $K\le G_\#$ and $G_\#$ is maximal.

  Conversely, take $g\in G_\#$ and $v\in X^*$; then $g\in G\prist{uv}$
  for all $u\in X^*$, so $v*g\in v*G\prist{uv}\le G\prist u$, so
  $v*g\in G_\#$, and $G_\#$ is a branching subgroup.
\end{proof}

A group $G\leq\aut X^*$ is \emph{recurrent} if $(\stab_G(v))\at{v}$ equals $G$ for all $v\in X^*$.
\begin{lem}
  Let $G$ be a recurrent, regular branch group. Then $G_\#\trianglelefteq G$.
\end{lem}
\begin{proof}
  We have $\rist_G(v)\trianglelefteq\stab_G(v)$, so $G\prist
  v\trianglelefteq\stab_G(v)@v=G$ for all $v\in X^*$.
\end{proof}

The following technical lemma has several important consequences.
\begin{lem}\label{lem:exists n}
  Let $G\leq\aut X^*$ be a self-similar, level-transitive, regular
  branch group.  Then there is an integer $n$ so that
  $G\prist{uv}=G\prist v$ for all $u,v\in X^*$ with $|v|\ge n$.
\end{lem}
\begin{proof}
  Since $G$ is self-similar, we have $G_\#\le G\prist{uv}\le G\prist
  v$ for all $u,v\in X^*$. Since $G$ is level-transitive,
  $[G\prist{v}:G_\#]$ depends only on $|v|$; call this index $a_m$ if
  $|v|=m$.  Then the $a_m$ form a non-increasing sequence of
  non-negative integers, so it must stabilize starting from some
  integer $n$.
\end{proof}

\begin{lem}
  For $v\in X^*$, the following properties are equivalent:
  \begin{enumerate}
  \item $G\prist v\le G\prist{uv}$ for all $u\in X^*$;
  \item $u*\rist_G(v)\le\rist_G(uv)$ for all $u\in X^*$;
  \item $\rist_G(v)\le G_\#$.
  \end{enumerate}
\end{lem}
\begin{proof}
  $(i)\Rightarrow(iii)$ because $\rist_G(v)=v*G\prist v\le
  v*G\prist{uv}\le G\prist u$ for all $u\in X^*$.

  $(iii)\Rightarrow(ii)$ because $G_\#\le G$ by definition, so
  $u*\rist_G(v)\le G_\#\le G$.

  $(ii)\Rightarrow(i)$ because $uv*G\prist
  v=u*\rist_G(v)\le\rist_G(uv)=uv*G\prist{uv}$.
\end{proof}

\begin{cor}\label{cor:rist cofinal to K}
  Let $G$ be a self-similar, level-transitive, regular branch
  group. Then there exists $n\ge0$ such that $X^m*G_\#\ge\rist_G(m+n)$
  for all $m\ge0$.
\end{cor}
\begin{proof}
  Let $n$ be given by Lemma~\ref{lem:exists n}. Then $\rist_G(v)\le
  G_\#$ for all $v\in X^n$. Since $G$ is self-similar, $\rist_G(uv)\le
  u*\rist_G(v)\le u*G_\#$ for all $u\in X^*$.
\end{proof}

\begin{cor}
  Suppose moreover that $G$ is recurrent. Then $X^m*G_\#=\rist_G(m)$
  for all $m\ge n$.
\end{cor}
\begin{proof}
  We have $G_\#=\bigcap_{v\in X^*}G\prist v$; and $G\prist
  v\trianglelefteq G$, and depends only on $|v|$; so the $G\prist v$
  are all equal to $G_\#$ when $|v|\ge n$. Then $\rist_G(w)=w*G\prist
  w=w*G_\#$ when $|w|\ge n$.
\end{proof}

We also define $\stab_G(n)=\stab(n)\cap G$.  These are normal
subgroups of $G$ and they form a filtration of $G$ with trivial
intersection.  Since $G/\stab_G(n)$ is finite for all $n$, the group
$G$ is residually finite.  Note that $\rist(n)=\stab(n)$ for the full
automorphism group of the rooted tree, but in general the inclusion
$\rist_G(n)\leq\stab_G(n)$ is proper.

\section{The congruence subgroup problem}
\subsection{General approach}
Consider a branch representation of an abstract group $\Gamma$ on a rooted tree $T$.
We consider three topologies on $\Gamma$.
We define these topologies by specifying a basis of neighborhoods of the identity as follows.
\begin{center}
\begin{tabular}{c|c|c}
basis of neighborhoods of $1$&name&completion\\
\hline\rule{0mm}{2.5ex}
$\{\stab_\Gamma(n)\mid n\in\N\}$&congruence topology&$\gb$\\
$\{\rist_\Gamma(n)\mid n\in\N\}$&branch topology&$\gt$\\
$\{N\mid N\trianglelefteq_f\Gamma\}$&profinite topology&$\gh$
\end{tabular}
\end{center}
The notation $N\trianglelefteq_f\Gamma$ means $N$ is a normal subgroup of finite index in $\Gamma$.
Since $\stab_\Gamma(n)$ contains $\rist_\Gamma(n)$ which has finite index in $\Gamma$, we deduce that the congruence topology is coarser than the branch topology, which is coarser than the profinite topology.
It might happen that some of these topologies coincide but in general they are all different.
These differences are characterized by the kernels of the natural maps:
\begin{align*}
\text{Congruence kernel}: \ker(\gh\to\gb)&=\bigcap_{n\geq0}\widehat{\stab_\Gamma(n)},\\
\text{Branch kernel}: \ker(\gh\to\gt)&=\bigcap_{n\geq0}\widehat{\rist_\Gamma(n)},\\
\text{Rigid kernel}: \ker(\gt\to\gb)&=\bigcap_{n\geq0}\widetilde{\stab_\Gamma(n)},
\end{align*}
where, for a subgroup $H$ of $\Gamma$, we denote by $\widehat{H}$ its closure in $\gh$, and similarly we
denote by $\widetilde{H}$ the closure of $H$ in $\gt$. Another way of
expressing this is the following:
\begin{align}
\ker(\gh\to\gb)&=\varprojlim(\stab_\Gamma(n)N)/N,\nonumber\\
\ker(\gh\to\gt)&=\varprojlim(\rist_\Gamma(n)N)/N,\label{eq:varprojlim}\\
\ker(\gt\to\gb)&=\varprojlim\stab_\Gamma(n)/\rist_\Gamma(n),\nonumber
\end{align}
where $n$ ranges over the non-negative integers and $N$ over the finite-index subgroups of $\Gamma$.
Consider $M\leq N$ two finite-index normal subgroups of $\Gamma$ and two integers $m\geq n$.
The connecting map $\stab_\Gamma(m)M/M\to\stab_\Gamma(n)N/N$ is given by
\begin{align*}
\stab_\Gamma(m)M/M\to\stab_\Gamma(n)N/M\to\stab(n)N/N
\end{align*}
where the first map is the inclusion and the second is the
quotient map. The other connecting maps in~\eqref{eq:varprojlim}
are defined analogously. We are in a position to state

\begin{prob}[\emph{Congruence subgroup problem}]
Compute the congruence, branch and rigid kernels.
\end{prob}
A preliminary question of great importance however is as follows:
Consider a branch group $\Gamma$.
\begin{quest}
Do any of the kernels depend on the branch representation of $\Gamma$?
\end{quest}

We shall address in this paper both of the questions.

\subsection{Branch kernel}
\subsubsection*{General properties}

We extracted the following lemma from the proof of~\cite{grigorchuk:jibg}*{Theorem~4}.
\begin{lem}\label{lem:rist}
Consider an element $g$ of a group $\Gamma\leq\aut T$ and let $v\leq w\in T$ be vertices so that $v$ is fixed by $g$, while $(w)^g\neq w$.
Then the normal closure of $g$ in $\Gamma$ contains $\rist_\Gamma(w)'$.
\end{lem}
\begin{cor}\label{cor:N contains rist'}
Let $\Gamma\leq\aut T$ be a level-transitive branch group and let $N\neq\{1\}$ be a normal subgroup of $\Gamma$.
Then $N$ contains $\rist_\Gamma(n)'$ for some $n\geq0$.
\end{cor}
\begin{proof}
Choose $1\neq g\in N$ and vertices $v$ and $w$ as in the lemma.
Then $N$ contains $\rist_{\Gamma}(w)'$.
Because $\Gamma$ is level-transitive, $N$ also contains $\rist_\Gamma(n)'$ where $n$ is the depth of $w$.
\end{proof}

We can drop the level-transitivity assumption under certain
circumstances.  We first need some definitions and one more lemma.
Consider a regular rooted tree $X^*$.  We say an element $g\in\aut
X^*$ is \emph{finite-state} if the set $\{g\at{w}\mid w\in X^*\}$ is
finite.  Observe that all the elements of a group generated by
finite-state elements are again finite-state.

\begin{lem}\label{lem:no fixed ray}
Let $\Gamma\leq\aut X^*$ be a weakly branch group generated by a finite number of finite-state elements.
Then $\Gamma$ does not fix any ray $w\in X^\omega$.
\end{lem}
\begin{proof}
  Let $S$ be a finite set of finite-state elements which generate
  $\Gamma$.  The set of fixed points of an automorphism $s\in S$ is a
  regular language in $X^*$.  The set $F$ of common fixed points of
  the elements of $S$, being the intersection of a finite number of
  regular languages, is again a regular language.

  Thus by the Pumping Lemma~\cite{hopcroft-u:automata}*{Chapter~3}, either
  $F$ is finite and we are done, or there is an integer $n$ such that
  any $w\in F$ of length at least $n$ can be decomposed as $w=uvz$
  with $|uv|\leq n$ and $|v|>0$, so that $uv^kz$ belongs to $F$ for
  all $k\geq0$.  This implies that the ray $uv^\omega$ is fixed by all
  the elements of $S$.  We show that this leads to a contradiction.

  Indeed, suppose that the elements of $S$, and thus $\Gamma$ itself,
  fix the ray $w=uv^\omega$.  Then there are positive integers $k$ and
  $n$ such that for each generator $s\in S$ we have
  $s\at{uv^m}=s\at{uv^{m+k}}$ for all $m\geq n$.  This is true for all
  the generators, hence it is true for the whole of $\Gamma$: for each
  element $g\in\Gamma$, we have $g\at{uv^m}=g\at{uv^{m+k}}$ for all
  $m\geq n$.  Clearly this contradicts the fact that $\Gamma$ is
  weakly branch.
\end{proof}

A subset $C$ of $X^*$ is a \emph{cut-set} if for each $w\in X^\omega$ there is exactly one element in $C$ which is a prefix of $w$.

\begin{prop}\label{prop:N contains rist'}
Let $\Gamma$ be a weakly branch group generated by a finite number of finite-state elements.
Then any subgroup $N$ of finite index in $\Gamma$ contains $\rist_\Gamma(n)'$ for some $n$.
\end{prop}
\begin{proof}
$N$ contains a normal subgroup of finite index, so we assume $N$ is normal in $\Gamma$.
First, we claim that there is a cut-set $C\subset X^*$ such that no element of $C$ is fixed by $N$.
Indeed, the non-existence of such a cut-set is equivalent to $N$ fixing a ray, but this is impossible by Lemma~\ref{lem:no fixed ray} (note that the hypotheses of the lemma are inherited by finite-index subgroups).

Next, for each element $w$ of $C$, we can find $g\in N$ which moves $w$.
Let $v$ be the longest prefix of $w$ which is fixed by $g$ and let $x$ be such that $vx$ is a prefix of $w$.
Then applying Lemma~\ref{lem:rist}, we see that $N$ contains $\rist_\Gamma(vx)'$, which in turn contains $\rist_\Gamma(w)'$.
Thus $N$ contains $\rist_\Gamma(n)'$ where $n$ is the length of the longest word in $C$.
\end{proof}

\begin{lem}\label{lem:N contains rist^e}
Let $\Gamma\leq\aut T$ be a branch group and let $N\neq\{1\}$ be a normal subgroup of finite index in $\Gamma$.
Then $N$ contains $\rist_\Gamma(n)^e$ for some $n,e\geq0$.
\end{lem}
\begin{proof}
The group $\Gamma/N$ is finite, hence it has finite exponent $e$.
In other words $\Gamma^e\leq N$, and since $\rist_\Gamma(n)$ is a subgroup of $\Gamma$ we have the inclusions $\rist_\Gamma(n)^e\leq\Gamma^e\leq N$ for all $n\geq0$.
\end{proof}

\begin{thm}\label{thm:branch kernel Abelian}
  Let $\Gamma\leq\aut T$ be a branch group; assume either that $G$ is
  level-transitive, or that $G$ is generated by finitely many
  finite-state elements.  The branch kernel $\K=\ker(\gh\to\gt)$ is
  Abelian.  Moreover, if there is a positive integer $e$ such that
  $\rist_\Gamma(n)/\rist_\Gamma(n)'$ is of exponent dividing $e$ for
  all $n\geq0$, then $\K$ is of exponent dividing $e$.
\end{thm}
\begin{proof}
  Either Corollary~\ref{cor:N contains rist'} or
  Proposition~\ref{prop:N contains rist'}, in combination with
  Lemma~\ref{lem:N contains rist^e}, prove that the filters $\{N\mid
  N\trianglelefteq_f\Gamma\}$ and
  $\{\rist_\Gamma(n)'\rist_\Gamma(n)^e\mid n,e\in\N\}$ are
  cofinal. Hence $\K$ is the inverse limit of the system
\begin{align*}
\rist_\Gamma(n)\big/(\rist_\Gamma(n)'\rist_\Gamma(n)^e).
\end{align*}
These groups are Abelian, hence the limit is Abelian.  If there is
$e\geq0$ such that $\rist_\Gamma(n)/\rist_\Gamma(n)'$ is of exponent
dividing $e$ for all $n\geq0$ then $\K$ is the inverse limit of the
system $\rist_\Gamma(n)/\rist_\Gamma(n)'$, hence $\K$ is of exponent
dividing $e$.
\end{proof}

Recall that if $\widetilde G=\varprojlim G/N$ is a profinite group and
$A$ is a ring, then $A[[\widetilde G]]=\varprojlim A[G/N]$ is a
profinite ring, called the \emph{complete group algebra} of
$\widetilde G$. See Sections 5.2, 5.3 and~5.5
of~\cite{ribes-z:pg} for more about free profinite modules,
complete group algebras and complete tensor products.

\begin{quotethm}[\ref{thm:description branch kernel}]
\descrbranchkernel
\end{quotethm}
\begin{proof}
The kernel $\K$ is the intersection $\bigcap_{n\geq0}\widehat{\rist_\Gamma(n)}$ by definition.
Since $\Gamma$ is level-transitive, $\rist_\Gamma(n)$ is the Cartesian
product of $\rist_\Gamma(v)$ for all $v$ of level $n$.
Define $A_w=\bigcap_{v<w}\widehat{\rist_\Gamma(v)}$, so that $\K$ embeds in the Cartesian product $\prod_{w\in\partial T}A_w$.
By Proposition~\ref{prop:N contains rist'} and Lemma~\ref{lem:N contains rist^e}, we have $A_w=\varprojlim_{v<w,k\geq0}\rist_\Gamma(v)/(\rist_\Gamma(v)'\rist_\Gamma(v)^k)$.
In particular, the group $A_w$ is Abelian and it can be
considered as a profinite ring.

Put $A_v:=\rist_\Gamma(v)/\rist_\Gamma(v)'$ and consider it as a ring.
Let $\Gamma_v$ be the stabilizer of the vertex $v$ in $\Gamma$.
Then $\rist_\Gamma(n)\big/(\rist_\Gamma(n))'$ can be
considered as a cyclic $A_v\Gamma$-module $\Z[\Gamma]\otimes_{\Z[\Gamma_v]}A_v$ with $v$ of length $n$. Therefore one has
$\K=\bigcap_{n\geq0}\widehat{\rist_\Gamma(n)}=\varprojlim\widehat{\rist_\Gamma(n)}=\varprojlim\rist_\Gamma(n)\big/(\rist_\Gamma(n)'\rist_\Gamma(n)^k)=
\varprojlim_{v<w,k\geq0}\Z[\Gamma]\otimes_{\Z[\Gamma_v]}(A_v/kA_v)$.
Using complete tensor products, one can rewrite this as
$\widehat\Z[[\gh]]\widehat\otimes_{\widehat\Z[[\gh_w]]}A_w$, where $\gh_w$ is the inverse limit $\bigcap_{v<w}\widehat{\Gamma_v}$.
Letting $\gh$ act on $T$ via the natural map $\gh\to\gb\leq\aut T$, the group $\gh_w$ is the stabilizer of $w$ in $\gh$.

Finally, by definition of $A_w$, the branch kernel acts trivially on it.
Therefore one can simplify $\widehat\Z[[\gh]]\widehat\otimes_{\widehat\Z[[\gh_w]]}A_w=\widehat\Z[[\gt]]\widehat\otimes_{\widehat\Z[[\gt_w]]}A_w$.
\end{proof}

\begin{rmk}
As an $A$-module, $\widehat\Z[[\gt]]\widehat\otimes_{\widehat\Z[[\gt_w]]}A$ is isomorphic to $A[[\gt/\gt_w]]\simeq A[[\partial T]]$.
However, in general, there might be an action of $\gt_w$ on $A$, as happens for example in Theorem~\ref{thm:congruence kernel Hanoi}.
\end{rmk}

\begin{cor}\label{cor:transitive regular branch}
  Let $\Gamma$ be a level-transitive, self-similar, regular branch
  group and let $K$ be its maximal branching subgroup.  The branch kernel
  $\K=\ker(\gh\to\gt)$ is isomorphic to
  $\widehat\Z[[\gt]]\widehat\otimes_{\widehat\Z[[\gt_w]]}A$ where $A$
  is the inverse limit of $\{K/(K'K^e)\}$ with the maps induced by
  the composition
\begin{align*} K/K'\simeq(K/K')\times\{1\}\times\cdots\times\{1\}\to(K/K')^d\to K/(K')^d\to K/K'.
\end{align*}
In particular, if $K/K'$ has finite exponent dividing $e$, then the branch kernel has exponent dividing $e$.
\end{cor}
\begin{proof}
Corollary~\ref{cor:rist cofinal to K} implies that the collections $\{X^n*K:n\geq0\}$ and $\{\rist_\Gamma(n):n\geq0\}$ are cofinal.
Therefore one can write the branch kernel as $\varprojlim (X^n*K)/(X^n*(K'K^e))$.
Rewriting the proof above starting from this point yields the corollary.
\end{proof}
\begin{rmk}\label{rmk:endomorphism}
There is always an endomorphism $\sigma$ of $K$ defined by $g\mapsto(g,1,\dots,1)$.
In the most favourable cases, $\sigma$ even extends to an endomorphism of the whole group $\Gamma$.
The map $\sigma$ induces an endomorphism of $K/K'$, and the connecting map above is just the inclusion $\sigma(K/K')\to K/K'$.
Thus one can rewrite the definition of $A$ as $A=\varprojlim_{e\geq0,n\geq0}\sigma^n(K/(K'K^e))$.
If $K/K'$ is finite, the previous expression simplifies to $A=\bigcap_{n\geq0}\sigma^n(K/K')$.
\end{rmk}

\begin{quotecor}[\ref{cor:branch kernel finite exponent}]
\branchkernelfinexpo
\end{quotecor}
\begin{proof}
Indeed $K'$ is normal in $\Gamma$, and hence has finite index in $\Gamma$, and \emph{a fortiori} in $K$.
The claim follows from Corollary~\ref{cor:transitive regular branch}.
\end{proof}

\subsubsection*{Independence of the representation}
We use a result of Ab\'ert~\cite{abert:nonfree}.
Let $G$ be a group acting on a set $X$.
We say that $G$ \emph{separates} $X$ if for any finite subset $Y\subseteq X$ the pointwise stabilizer $G_Y$ does not stabilize any point outside $Y$.
Recall that a group $G$ \emph{satisfies a group law} if there is a non-trivial element $w$ in a free group $F$ which maps to the identity under any homomorphism $F\to G$.
\begin{thm}[\cite{abert:nonfree}*{Theorem 1}]\label{thm:abert}
If a group $G$ separates $X$ then $G$ does not satisfy any group law.
\end{thm}
\begin{cor}\label{cor:no law in branch groups}
Let $\Gamma\leq\aut T$ be a level-transitive branch group and let
$N$ be a non-trivial normal subgroup of $\Gamma$. Then $N$ cannot
satisfy any group law.
\end{cor}
\begin{proof}
If a non-trivial normal subgroup $N$ of  $\Gamma$ satisfies a law,
then $\Gamma$ satisfies a law. Indeed by Corollary~\ref{cor:N contains rist'} the group $N$ contains the commutator subgroup of the rigid
stabilizer of some level, so this rigid stabilizer must satisfy a
law. However, it is of finite index in $\Gamma$, so $\Gamma$ satisfies
a law. This shows that in our case it is sufficient to prove that
$\Gamma$ satisfies no law.

$\Gamma$ acts on the boundary $\partial T$ of $T$ and by
Theorem~\ref{thm:abert} it is enough to show that $\Gamma$ separates
$\partial T$.  Consider a finite subset $Y\subset\partial T$ and a
point $w\in\partial T\setminus Y$.  Choose a prefix $v$ of $w$ such
that no descendant of $v$ is in $Y$.  Then $\rist_\Gamma(v)$ is
non-trivial because $\Gamma$ is branch and level-transitive, thus we
can choose a non-trivial element $g\in\rist_\Gamma(v)$.  The
automorphism $g$ moves a vertex $u\geq v$.  Conjugating $g$ by an
element $h\in\Gamma$ which moves $u$ to a prefix of $w$ gives an
element $g'$ in $\rist_\Gamma(v)$ which moves $w$, as needed.
\end{proof}

We are now able to prove Theorem~\ref{thm:independence branch kernel}, which shows that for a level-transitive, branch representation, the branch kernel is contained in the congruence kernel of any other level-transitive, branch representation.

We do not know any example of a group having representations with different rigid kernels.
\begin{quotethm}[\ref{thm:independence branch kernel}]
\indepbranchkernel
\end{quotethm}
\begin{proof}
$\K_i$ is a normal subgroup of $\gh$ and it is Abelian by Theorem~\ref{thm:branch kernel Abelian}.
The image $\overline{\K}_1^{(2)}$ of $\K_1$ in $\gb_2$ is a normal subgroup and it is Abelian too.
Because of Corollary~\ref{cor:no law in branch groups} the group $\overline{\K}_1^{(2)}$ is trivial because $\gb_2$ is branch.
Therefore $\K_1$ lies in the kernel of the map $\gh\to\gb_2$.
The argument is clearly symmetric, and so $\K_2$ is in the kernel of the map $\gh\to\gb_1$.
\end{proof}

\subsection{Rigid kernel}
Here we only consider the case of self-similar, regular branch groups.
\subsubsection*{General properties}
\begin{thm}\label{thm:rigid kernel bounded}
  Let $\Gamma$ be a level-transitive, self-similar, regular branch
  group.  Then the kernel of $\gt\to\gb$ has finite exponent.
\end{thm}
\begin{proof}
  Let $K$ be the maximal branching subgroup of $\Gamma$.  By
  Corollary~\ref{cor:rist cofinal to K} the collections
  $\{X^n*K:n\geq0\}$ and $\{\rist_\Gamma(n):n\geq0\}$ are cofinal.
  Therefore the rigid kernel can be computed as
  $\varprojlim\stab_\Gamma(n)/\core_\Gamma(X^n*K)$, where for a subgroup $H\le \Gamma$ we define $\core_\Gamma(H)=\bigcap_{g\in \Gamma}H^g$.
    
  Note that $\core_\Gamma(X^n*K)$ contains $X^n*\core_\Gamma(K)$ because the latter is contained in $X^n*K$ and normal in $\Gamma$.
    The image of $\stab_\Gamma(n)$
  through the wreath decomposition
  $\Gamma\to\Gamma^{X^n}\rtimes(\Gamma/\stab_\Gamma(n))$ is contained
  in $\Gamma^{X^n}$. Therefore $\stab_\Gamma(n)/\core_\Gamma(X^n*K)$ is a subquotient of $(\Gamma/\core_\Gamma(K))^{X^n}$ for all $n\geq0$.  This
  proves that the rigid kernel of $\Gamma$ has exponent bounded by $e$
  where $e$ is the exponent of $\Gamma/\core_\Gamma(K)$, which is finite by Lemma~\ref{lem:(regular) branch} and Corollary~\ref{cor:rist cofinal to K}.
\end{proof}

\subsubsection*{Independence of the representation}
\begin{quotethm}[\ref{thm:regular branch -> independence}]
\regbranchindep
\end{quotethm}
\begin{proof}
Combining Theorems~\ref{thm:branch kernel Abelian} and~\ref{thm:rigid kernel bounded} we see that $\K_1$ and $\K_2$ are both extensions of an Abelian group by a group of finite exponent.
Thus $\K_1$ and $\K_2$ satisfy a group law (big powers commute).
Corollary~\ref{cor:no law in branch groups} shows that the image of $\K_1$ in $\gb_2$ is trivial, thus $\K_1\leq\K_2$ and vice versa.
\end{proof}

\section{Examples}
We describe in Sections~\ref{sec:pervova} and~\ref{sec:twin} some
groups with trivial rigid kernel, but with non-trivial branch kernel.
In contrast, we describe in Section~\ref{sec:Hanoi} the Hanoi tower
group, which has both branch and rigid kernel non-trivial.

\subsection{Pervova's examples}\label{sec:pervova}
Pervova~\cite{pervova:profinitecompletions} has constructed the first examples of groups acting on rooted trees which fail to have the congruence subgroup property, i.e. the kernel of $\gh\to\gb$ is non-trivial.
Her examples act on a $p$-ary rooted tree where $p$ is an odd prime.
We detail the congruence kernels only for $p=3$ for simplicity but the other examples are similar.

Define the rooted tree automorphisms $a,b,c$ via the recursions
\begin{align*}
a&=\sigma_{123},&b&=(a,a^{-1},b),&c&=(c,a,a^{-1}).
\end{align*}
We let $\Gamma$ be the group generated by $a,b,c$. Considering
furthermore the element $d=(d,d,d)\sigma_{123}$, we see that $\langle
a,c\rangle=\langle a,b\rangle^d$ is the infinite torsion group
considered by Gupta and Sidki~\cite{gupta-s:infinitep}. In fact, the
same proof shows that $\Gamma$ is a torsion group.
Clearly the group $\Gamma$ is self-similar, level-transitive and recurrent.
\begin{lem}
  $\Gamma$ is regular branch, with branching subgroup $\Gamma'$.
\end{lem}
\begin{proof}
  It suffices to check the computations
  \begin{align*}
    1*[a,b] &= [b^{-a}b^{-a^2},b^{-1}b^{-a}],\\
    1*[a,c] &= [c^{-1}c^{-a},c^{-a^2}c^{-1}],\\
    1*[b,c] &= [b^a,c].\qedhere
  \end{align*}
\end{proof}

We have~\cite{pervova:profinitecompletions}*{Lemma 1.4}:
\begin{align*}
\Gamma/\Gamma'\simeq C_3\times C_3\times C_3,
\end{align*}
but on the other hand it is easy to see that $\gb/\gb'\simeq C_3\times
C_3$; indeed $c^{-1}b=[a,b](1*[a,b])(1^2*[a,b])\cdots$ belongs to $\gb'$.
This proves that $\Gamma'$ is not a congruence subgroup.

\begin{prop}[\cite{pervova:profinitecompletions}*{Proposition 1.3}]
We have the inclusion $\Gamma'\leq\Gamma\prist{v}$ for all vertices $v$.
\end{prop}
From this we show slightly more.
Let $K$ be the group generated by $\Gamma'$ and $b^{-1}c$.
\begin{prop}
We have the inclusion $K\leq\Gamma\prist{v}$ for all vertices $v$.
\end{prop}
\begin{proof}
It is enough to show that $K$ contains $1*(b^{-1}c)$.
A direct computation shows $[a,b]b^{-1}c=b^{-a}c=1*(b^{-1}c)\in K$.
\end{proof}

\begin{prop}\label{prop:PSK}
The groups $\gt$ and $\gb$ coincide.
\end{prop}
\begin{proof}
We show that $K$ contains $\stab_\Gamma(2)$. This proves that the
topologies induced by $\{\stab_\Gamma(n)\mid n\in\N\}$ and
$\{\rist_\Gamma(n)\mid n\in\N\}$ are the same.

Clearly $K$ has index $9$ in $\Gamma$, and a direct computation shows that $\pi_2(\Gamma)/\pi_2(K)$ is a group of order $9$, where $\pi_2$ is the canonical map $\Gamma\to\Gamma/\stab_\Gamma(2)$.
Thus $K$ contains $\stab_\Gamma(2)$.
\end{proof}

\begin{thm}[\cite{pervova:profinitecompletions}*{Theorem 3.2}]
Any finite index subgroup $N$ of $\Gamma$ contains $X^n*\Gamma'$ for some $n$.
\end{thm}

In~\cite{pervova:profinitecompletions}*{Theorem 3.4}, the congruence kernel of $\Gamma$ is described as an inverse limit of elementary Abelian $p$-groups.
Using our Corollary~\ref{cor:transitive regular branch}, the structure of the kernel follows from an easy computation, and we get the $\gb$-module action almost for free.
\begin{cor}
  The branch and congruence kernels of $\Gamma$ are
  $(\Z/3\Z)[[X^\omega]]$.
\end{cor}
\begin{proof}
  Both kernels are equal by Proposition~\ref{prop:PSK}, so we
  concentrate on the former.  We use Corollary~\ref{cor:transitive
    regular branch} to deduce that the congruence kernel is
  $\widehat\Z[[\gt]]\widehat\otimes_{\widehat\Z[[\gt_w]]}A$ where $A$
  is the inverse limit of $K/\Gamma'$ with the connecting map
\begin{align*}
K/\Gamma'\to K/\Gamma'\times\{1\}\times\{1\}\to K/(\Gamma')^3\to K/\Gamma'.
\end{align*}
$K/\Gamma'$ is a cyclic group of order $3$ generated by $b^{-1}c$, and a straightforward computation shows $1*(b^{-1}c)=[a,b]b^{-1}c\equiv b^{-1}c\pmod{\Gamma'}$.
Therefore $A$ is the inverse limit of $\Z/3\Z$ with the connecting map being the identity, whence $A\simeq \Z/3\Z$.
Moreover the action of $\stab_\Gamma(v)$ on $v*(K/\Gamma')$ is trivial.
Therefore $\widehat\Z[[\gt]]\widehat\otimes_{\widehat\Z[[\gt_w]]}A=(\Z/3\Z)[[\gt/\gt_w]]=(\Z/3\Z)[[\gb/\gb_w]]=(\Z/3\Z)[[X^\omega]]$.
\end{proof}

\subsection{The twisted twin of Grigorchuk's group}\label{sec:twin}
We sketch here the computation of the congruence kernel for a new
example of group, which is a twisted relative of Grigorchuk's first
group~\cite{grigorchuk:burnside}. More details will appear
in~\cite{bartholdi-s:tt}.

We define the automorphisms of the binary tree
\begin{align*}
a&=\sigma_{12},&b&=(c,a)\\
c&=(a,d),&d&=(1,b).
\end{align*}
We let $H$ be the group generated by $a,b,c,d$.  The group $H$ is a
just-infinite, torsion, level-transitive, self-similar, recurrent,
regular branch group.

Define $K$ as the normal closure of $\{[a,b],[b,c],[c,d],[d,b],bcd\}$ in $H$.
We have
\begin{prop}[\cite{bartholdi-s:tt}]
The group $K$ contains $K\times K$, and $K$ contains $\stab_H(3)$.
\end{prop}
\begin{cor}\label{cor:TSK}
The groups $\widetilde{H}$ and $\overline{H}$ coincide.
\end{cor}

\begin{prop}[\cite{bartholdi-s:tt}]\label{prop:eviltwin fi}
  Any finite-index normal subgroup $N$ of $H$ contains $X^n*[K,H]$
  for some $n\geq0$.
\end{prop}

Let $\sigma$ be the endomorphism of $H$ induced by $a\mapsto c, b\mapsto d^a, c\mapsto b, d\mapsto c^a$.
\begin{thm}\label{thm:eviltwin kernel}
  The branch and congruence kernels of $H$ are
  $(\mathbb{Z}/4\mathbb{Z})[[X^\omega]]$.
\end{thm}
\begin{proof}
  Both kernels are equal by Corollary~\ref{cor:TSK}, so we concentrate
  on the former.  We shall compute the group $A$ from
  Corollary~\ref{cor:transitive regular branch} using an improvement
  of Remark~\ref{rmk:endomorphism}.  Namely, we have an endomorphism
  $\sigma$ of $H$ inducing $g\mapsto(g,1)$ on $K$.  Moreover, by
  Proposition~\ref{prop:eviltwin fi}, we only have to consider the
  group $K/[K,H]$.  This group inherits the endomorphism $\sigma$, and
  thus we have $A=\bigcap_{n\geq0}\sigma^n(K/[K,H])$.  It is proved
  in~\cite{bartholdi-s:tt} that the group $K/[K,H]$ is isomorphic to $C_4\times
  C_2$, the $C_4$ being generated by $bcd$ and the $C_2$ by $[a,b]$.
  Further computations show that $\sigma$ sends $bcd$ to $bcd[a,b]$
  and $[a,b]$ to $(bcd)^2$ mod $[K,H]$.  Thus $\sigma(K/[K,H])$ is a
  cyclic group of order $4$, generated by $bcd[a,b]$, and it is
  invariant under $\sigma$ (it is not fixed pointwise, however).

Finally, the action by conjugation of $H$ on $K/[K,H]$ is obviously trivial.
\end{proof}

\subsection{The Hanoi tower group}\label{sec:Hanoi}
This group was introduced by Sunic and Grigorchuk;
see~\cite{grigorchuk-s:hanoi} or~\cite{boston-c-g:prop}*{page 1477}.
It models the ``Towers of Hanoi'' game in the sense that the graph of
the action of $\Gamma$ on $X^n$ describes the space of moves of $n$
disks on a $3$-peg game.

We define the following automorphisms of the $3$-ary rooted tree:
\begin{align*}
a&=(a,1,1)\sigma_{23},&b&=(1,b,1)\sigma_{13},&c&=(1,1,c)\sigma_{12}.
\end{align*}
We let $\Gamma$ be the group generated by $a,b,c$.
Since the definition is symmetric in $a,b,c$, any permutation of the letters $a,b,c$ yields an automorphism of $\Gamma$.

The Abelianization $\Gamma/\Gamma'$ of $\Gamma$ is isomorphic to $C_2\times C_2\times C_2$, but we have $\gb/\gb'\simeq C_2$.
Therefore $\Gamma'$ is not a congruence subgroup.
However, the situation is more subtle than in the previous examples.

\begin{thm}\label{thm:HCK}
\begin{itemize}
\item
The kernel of $\gh\to\gt$ is free profinite Abelian;
\item
The kernel of $\gt\to\gb$ is a Klein group of order 4;
\item
The kernel of $\gh\to\gb$ is metabelian and torsion-free, but is not nilpotent.
\end{itemize}
\end{thm}

\subsubsection*{A presentation of $\Gamma$}

The main tool for this section is a recursive presentation of
$\Gamma$.

We recall the general strategy in obtaining presentations by generators
and relations for self-similar groups; for more details
see~\cite{bartholdi:lpres} or \cite{sidki:pres}.

We note that $\Gamma$ is contracting with nucleus $N=\{1,a,b,c\}$.
By definition, this means that for any $g\in\Gamma$, there is an $n\geq0$ so that $g\at{w}$ is in $N$ for all $w$ of length at least $n$. The
only relations of length $\le3$ among elements of $N$ are
$a^2=b^2=c^2=1$. We consider thus the group $F=\langle
a,b,c|a^2,b^2,c^2\rangle$. The decomposition map
$\Gamma\to\Gamma\wr\sym(X)$ restricts to a map $N\to N^X\times\sym(X)$,
which induces a homomorphism $\psi:F\to F\wr\sym(X)$. Set
$K_0=1\triangleleft F$ and $K_{n+1}=\psi^{-1}(K_n^X)$ for all $n\ge0$.
\begin{lem}
  $\Gamma=F/\bigcup_n K_n$.
\end{lem}
\begin{proof}
  By our choice of relations in $F$, the decomposition $\psi$ is
  contracting on $F$, with nucleus $\{1,a,b,c\}$. Given $w\in F$: if
  $w$ belongs to $K_n$ for some $n$, then it is clear that $w$ is
  trivial in $\Gamma$. Conversely, if $w$ is trivial in $\Gamma$,
  there exists $n$ such that all $w$'s level-$n$ states belong to $N$
  and act trivially; so they are all $1$; so $w$ belongs to $K_n$.
\end{proof}

It is easy, using the Reidemeister-Schreier rewriting method, to
construct a normal generating set for $K_1$.  Indeed, $\psi$ induces
an injective map $F/K_1\to F\wr\sym(X)$, whose image has index
$32$. We introduce some notation:
\[d=[a,b],\quad e=[b,c],\quad f=[c,a],\quad g=d^c,\quad h=e^a,\quad i=f^b.\]
Then we obtain quite explicitly
\[K_1=\langle d^{-1}efi^{-1}ge^{-1}, he^{-1}d^{-1}fdi^{-1}, e^{-1}g^{-1}f^{-1}egf, e^{-1}dhe^{-2}d^{-1}h^2, hgd^{-2}f^{-1}gfe^{-1}\rangle^F.\]
We also consider the homomorphism $\tau:F\to F$ defined by
\[a\mapsto a,\quad b\mapsto b^c,\quad c\mapsto c^b.\]
\begin{lem}
  Let $\tau'$ be any homomorphism $F\to D_\infty=\langle
  x,y|x^2,y^2\rangle$ that sends $a$ to $1$ and $b,c$ to conjugates of
  $x,y$ respectively. Then $\tau'(K_1)=1$.
\end{lem}
\begin{proof}
  We have $\tau'(d)=\tau'(f)=\tau'(g)=\tau'(i)=1$, and
  $\tau'(e)=\tau'(h)$. It remains to check that each normal generator of
  $K_1$ contains (with sign) as many $e$'s as $h$'s.
\end{proof}

It follows that, for all normal generators $r$ of $K_1$, we have
$\psi(\tau(r))=(r,1,1)$. Therefore, $K_n$ is normally generated by
$\bigcup_{i<n}\tau^i(K_1)$ for all $n>0$. We conclude:
\begin{prop}\label{prop:pres gamma}
  \[\Gamma=\langle
  a,b,c\mid a^2,b^2,c^2,\tau^n(w_1),\tau^n(w_2),\tau^n(w_3),\tau^n(w_4),\tau^n(w_5)\text{
    for all }n\ge0\rangle,\]
  where $w_1,\dots,w_5$ are the five normal generators of $K_1$ above.
\end{prop}
\begin{rmk}
  One can check that $w_5$ is a consequence of
  $w_1,w_2,w_3,w_4,\tau(w_1)$ in $F$, and thus one has the simpler
  presentation
\[\Gamma=\langle a,b,c\mid
a^2,b^2,c^2,\tau^n(w_1),\tau^n(w_2),\tau^n(w_3),\tau^n(w_4)\text{ for
  all }n\ge0\rangle.\]
\end{rmk}

We deduce from Proposition~\ref{prop:pres gamma} a short proof of a
result by Sunic:
\begin{prop}[Sunic]\label{prop:hanoi endomorphism}
  The map $a\mapsto1$, $b\mapsto b$, $c\mapsto c$ extends to an
  endomorphism of $\Gamma$.
\end{prop}
\begin{proof}
  This map is an instance of a map $\tau'$ in the lemma; all relators
  of $\Gamma$ are mapped to one under it.
\end{proof}

\subsubsection*{The branch kernel of $\Gamma$}

\begin{prop}
  The group $\Gamma'/\Gamma''$ is isomorphic to $\Z^3\times C_3$,
  where $\Z^3$ is generated by $[a,b],[b,c],[c,a]$ and $C_3$ is
  generated by $[[a,b],c]$.
\end{prop}
\begin{proof}
  $\Gamma'$ is generated by $d,e,f,g,h,i$ in the notation introduced
  above. The exponent counts of these generators in the relations
  $r_1,\dots,r_5$ are given by the matrix
  \[\begin{pmatrix}
   -1 &   & 1 & 1 &   &-1\\
      &-1 & 1 &   & 1 &-1\\
    & & & & &\\
    &-3 & & & 3 &\\
    -2&-1&&2&1&
  \end{pmatrix}
  \rightsquigarrow\cdots\rightsquigarrow
  \begin{pmatrix}
  1&&2&-1&&-2\\
  &1&2&&-1&-2\\
  &&3&&&-3
  \end{pmatrix};\] its Smith normal form has two ones and a three on
  its diagonal; the three generators of infinite rank of
  $\Gamma'/\Gamma''$ may be chosen as $d,e,f$, while $eh^{-1}$
  generates a $C_3$ therein, and equals $dg^{-1}$ and $fi^{-1}$.
\end{proof}

\begin{prop}\label{prop:branch kernel hanoi}
The branch kernel of $\Gamma$ is isomorphic to $\Zh^3[[X^\omega]]$.
\end{prop}
\begin{proof}
We consider the group $\Gamma'/\Gamma''\simeq\Z^3\times C_3$ with the endomorphism defined by the composition
\begin{align*}
\Gamma'/\Gamma''\xrightarrow{\sim}\Gamma'/\Gamma''\times\{1\}\times\{1\}\to\Gamma'/(\Gamma'')^3\to\Gamma'/\Gamma'',
\end{align*}
and we would like to compute the limit $\varprojlim\Gamma'/\Gamma''$.
Recall that $dg^{-1}=eh^{-1}=fi^{-1}$ mod $\Gamma''$, and note that in $\Gamma$, the relation $(dg^{-1},1,1)=h^{-1}de^3g^{-1}fh^{-1}f^{-1}e^{-1}$ holds.
This element is equal to $dg^{-1}e^2h^{-2}=1$ mod $\Gamma''$.

Next, we compute $(d,1,1)=h^{-1}de$, and similarly $(e,1,1)=e^{-1}if^{-1}$ and $(f,1,1)=efh^{-1}$ in $\Gamma$.
These elements are equal to $deh^{-1}$, $e^{-2}h$ and $feh^{-1}$ respectively, mod $\Gamma''$.
This shows that the endomorphism of $\Gamma'/\Gamma''$ is given by left multiplication by the matrix
\begin{align*}
\begin{pmatrix}
 1 & 0 & 0 & 0 \\
 0 &-1 & 0 & 0 \\
 0 & 0 & 1 & 0 \\
 1 &-1 & 1 & 0 
\end{pmatrix}
\end{align*}
with rows and columns indexed by $\{d,e,f,eh^{-1}\}$. Therefore, using
Corollary~\ref{cor:transitive regular branch}, we have proved the
proposition.
\end{proof}

\subsubsection*{The rigid kernel of $\Gamma$}
Let $A$ be the group $\Gamma/\Gamma'\simeq(\Z/2\Z)^3$, generated by $\bar a,\bar b,\bar c$.
Viewing $A$ as a ring (with component-wise multiplication), we denote by $A[[X^\omega]]$ the free profinite $A$-module on the profinite space $X^\omega=\varprojlim X^n$.
Then $A[[X^\omega]]$ is in particular an Abelian profinite group, and it is endowed with a natural action of $\aut X^*$.
If $m=\sum_{w\in X^\omega}w*m_w$ is an element of $A[[X^\omega]]$ and $g$ is in $\aut X^*$, then the action of $g$ on $m$ is given by
\begin{align*}
g\cdot m=\sum_{w\in X^\omega}w^{g^{-1}}*m_w.
\end{align*}
Thus we can consider the extension $W=A[[X^\omega]]\rtimes\aut X^*$ which is again a profinite group.
It is straightforward to check that $W$ is isomorphic (as a profinite
group) to the inverse limit of the groups $W_n=A[X^n]\rtimes\aut X^n$
with the natural maps $W_{n+1}\to W_n$ given by
\[\left(\sum_{w\in X^n,x\in X}wx*m_{wx},g\right)\mapsto\left(\sum_{w\in X^n,x\in X}w*m_{wx},g|X^n\right).\]

We define the following elements of $W$:
\begin{align*}
\alpha&=(1^\omega *\bar a,a),&\beta&=(2^\omega *\bar b,b),&\gamma&=(3^\omega *\bar c,c),
\end{align*}
and we let $G$ be the subgroup of $W$ generated by $\alpha,\beta,\gamma$.
\begin{prop}
The natural epimorphism $W\to\aut X^*$ restricts to an isomorphism $\pi:G\to\Gamma$.
\end{prop}
\begin{proof}
The map $\pi$ is obviously an epimorphism.
To complete the proof, we construct an inverse $\psi:\Gamma\to G$.
To this end we set $\psi(g)=(\psi_M(g),g)$ and we still have to define $\psi_M(g)$.
We first set
\begin{align*}
\psi(a)&=\alpha,&\psi(b)&=\beta,&\psi(c)&=\gamma,
\end{align*}
and $\psi_M(1)=0$.
It is readily checked that the formula
\begin{align*}
\psi_M(g)=\sum_{w\in X^n}w*\psi_M(g\at{w})
\end{align*}
holds for all $n\geq0$ and $g\in N$.
By imposing it to hold for all $g$ in $\Gamma$, the formula defines $\psi_M(g)$ for all $g\in\Gamma$ because of the contracting property.
Indeed for an element $g\in\Gamma$, we can choose $n$ big enough so that $g\at{w}$ is in $N$ for all $w\in X^n$.
Then $\psi_M(g)=\sum_{w\in X^n}w*\psi_M(g\at{w})$ is well-defined.
We have thus defined $\psi(g)=(\psi_M(g),g)$ for all $g\in\Gamma$.

Obviously we have $\pi\circ\psi=\id_\Gamma$, and the composition $\psi\circ\pi$ is the identity when restricted to $\psi(N)$.
Since $\psi(N)$ generates $G$, we will show by induction on the length of $g$ that the relation $\psi(\pi(g))=g$ holds for all $g\in G$ (we consider the length of $g$ in the word metric of $G$ with respect to the generating set $\psi(N)$).

Consider $g=(m,\bar g)\in G$ and $h=(n,\bar h)\in\psi(N)$.
Then we have $\pi(gh)=\bar g\bar h$ and we must show $\psi(\bar g\bar h)=gh=(m+\bar g\cdot n,\bar g\bar h)$.
Thus all we need to prove is $\psi_M(\bar g\bar h)=m+\bar g\cdot n$.
Choose $k$ big enough so that $\bar g\at{w}$ and $(\bar g\bar h)\at{w}$ all belong to $N$ for all $w\in X^k$.
Then by definition of $\psi$ we have
\begin{multline*}
\psi_M(\bar g\bar h)=\sum_{w\in X^k}w*\psi_M((\bar g\bar h)\at{w})=\sum_{w\in X^k}w*\psi_M(\bar g\at{w}\bar h\at{w^{\bar g}})=\sum_{w\in X^k}w*(\psi_M(\bar g\at{w})+\psi_M(\bar h\at{w^{\bar g}}))\\=\sum_{w\in X^k}w*\psi_M(\bar g\at{w})+\sum_{w\in X^k}w^{\bar g^{-1}}*\psi_M(\bar h\at{w})=m+\bar g\cdot n.
\end{multline*}
For the third equality we used the fact that the relation $\psi_M(gh)=\psi_M(g)+\psi_M(h)$ holds when $g$, $h$ and $gh$ are all in $N$.
For the last equality we used the relation $\psi_M(h)^{g^{-1}}$, which also holds when $g$, $h$ and $gh$ are all in $N$.
\end{proof}

We now identify $G$ and $\Gamma$ using the above isomorphism.
\begin{prop}\label{prop:branch completion Hanoi}
The closure of $\Gamma$ in $W$ is isomorphic to the branch completion of $\Gamma$.
\end{prop}
\begin{proof}
In the group $\Gamma$, a fundamental system of neighborhoods of $1$ is
$\{X^n*\Gamma'\mid n\geq0\}$ for the branch topology, and
$\{N_n\mid n\geq0\}$ for the subspace topology in $W$.
Here $N_n$ denotes the kernel of the canonical map $W\to W_n$, intersected with $\Gamma$.
We need to show that these two filters are cofinal.

In fact we have the equality $N_n=X^n*\Gamma'$ for all $n\geq0$.
Indeed, notice that we have the isomorphism
$\Gamma/(X^n*\Gamma')\simeq(\stab_\Gamma(n)/(X^n*\Gamma'))\rtimes(\Gamma/\stab_\Gamma(n))$,
and $\stab_\Gamma(n)/(X^n*\Gamma')$ is a subgroup of
$(\Gamma/\Gamma')^{X^n}\simeq A[X^n]$.  Thus
$\Gamma\to\Gamma/(X^n*\Gamma')$ and $\Gamma\to W_n$ are in fact one
and the same map.  Hence their kernels coincide.
\end{proof}

Further, we consider the ring $R$ of continuous functions $W\to\Z/2\Z$.
In particular we define the following elements of $R$.
Consider a subset $S\subseteq\sym(X)$.
We define $\chi_S:W\to\Z/2\Z$ as the composition of $W\to\aut X^*\to\sym(X)$ with the characteristic function $\sym(X)\to\Z/2\Z$ of $S$.
For example if $S=\{(1,2),(2,3),(3,1)\}$ then $\chi_S$ is the homomorphism $[\sigma]:W\to\Z/2\Z$ which gives the signature of the permutation on top of the tree.
We also define the function $[a]:A[[X^\omega]]\to\Z/2\Z$ as the composition $A[[X^\omega]]\to A[X^0]\simeq A$ with the functional $[a]:A\to\Z/2\Z$ which takes the value $1$ on $\bar a$ and $0$ on $\bar b$ and $\bar c$.
We extend this function to $W$ by $[a](m,g)=[a](m)$.
We define similarly $[b]$ and $[c]$.
Next we define the operation $w*\cdot:R\to R$ by $(w*f)(m,g)=f(m\at{w},g\at{w})$ for all $w\in X^*$ (where $m\at{w}$ is the coefficient of $m$ with index $w$ in the canonical decomposition $A[[X^\omega]]\simeq\bigoplus_{w\in X^n}A[[X^\omega]]$).
We extend this definition by linearity on the left factor to obtain an operation $*:(\Z/2\Z)[X^*]\otimes R\to R$.

Now we define some specific elements of $R$.
\begin{align*}
P_1&=\chi_{\sym(X)\setminus\stab(1)}+(2+3)*([b]+[c])=\chi_{\sym(X)\setminus\stab(1)}+2*[b]+2*[c]+3*[b]+3*[c],\\
Q_{1,2}&=\chi_{(1,2)\cdot\stab(1)}+(1+3)*([b]+[c])=\chi_{(1,2)\cdot\stab(1)}+1*[b]+1*[c]+3*[b]+3*[c],\\
Q_{1,3}&=\chi_{(1,3)\cdot\stab(1)}+(2+1)*([b]+[c])=\chi_{(1,3)\cdot\stab(1)}+2*[b]+2*[c]+1*[b]+1*[c],
\end{align*}
and we define similarly $P_2,P_3,Q_{2,1},Q_{2,3},Q_{3,1},Q_{3,2}$.
One can check that $Q_{1,2}+Q_{1,3}=P_1$ (and similarly for $P_2,P_3$).
We define $\Sigma=[\sigma]+[a]+[b]+[c]$.
Then one has $(\varnothing+1+2+3)*\Sigma=(\varnothing+1+2+3)*[\sigma]$ because $(\varnothing+1+2+3)*[a]=[a]$ and similarly for $b$, $c$.

A subset $S\subseteq W$ is called \emph{self-similar} if $s\at{w}\in S$ for
all $s\in S$, $w\in X^*$.  As one does in algebraic geometry, we write
$\V(S)$ for the set of elements $g$ of $W$ so that $P(g)=0$ for all
$P\in S$. Define
\[\gt=\V(X^**\{\Sigma,P_1,P_2,P_3,Q_{1,2},Q_{2,3},Q_{3,1}\}).\]
We prove:
\begin{thm}\label{thm:BC}
  The group $\gt$ is the branch completion of $\Gamma$.
\end{thm}

Note that, for now, it is not even clear that $\gt$ is a group. We
prove Theorem~\ref{thm:BC} in the following lemmata.

\begin{lem}
  The set $\gt$ is a closed self-similar subgroup of $W$.
  Moreover, $\Gamma$ is a subgroup of $\gt$.
\end{lem}
\begin{proof}
Obviously, $\gt$ is self-similar and closed.
We need to prove it is a group.
Using the self-similarity, it is enough to prove that the projection of $\V(\{\Sigma,P_1,P_2,P_3,Q_{1,2},Q_{2,3},Q_{3,1}\})$ into $W_1$ is a subgroup of $W_1$.
This is can be done by hand or using a computer, but is in any case straightforward.

Now $\Gamma$ is clearly a subgroup of $\gt$ because $a$, $b$ and $c$ are in $\gt$.
\end{proof}
\begin{rmk}
One can check the relations $P_1+P_2+P_3=Q_{1,2}+Q_{2,1}=Q_{2,3}+Q_{3,2}=Q_{3,1}+Q_{1,3}$.
Thus one has $\widetilde{\Gamma}=\V(X^**\{\Sigma,Q_{1,2},Q_{1,3},Q_{2,1},Q_{2,3}\})$, and one can check that $X^**\{\Sigma,Q_{1,2},Q_{1,3},Q_{2,1},Q_{2,3}\}$ is a minimal generating set for the ideal $I\subset R$ of functions vanishing on $\widetilde{\Gamma}$.
\end{rmk}
\begin{lem}
Let $H$ be the projection of $\widetilde{\Gamma}$ to $\aut X^*$.
The group $\stab_H(n)/\stab_H(n+1)$ has $6$ elements when $n=0$, and $2^{2\cdot3^{n-1}}\cdot3^{3^n}$ elements for all $n\geq1$.
Moreover, the same holds for $\stab_\Gamma(n)/\stab_\Gamma(n+1)$.
\end{lem}
\begin{proof}
We first prove that $\stab_H(n)/\stab_H(n+1)$ has at most $2^{2\cdot3^{n-1}}\cdot3^{3^n}$ elements for all $n\geq1$.
By definition of $\widetilde{\Gamma}$, the function $(\varnothing+1+2+3)*\Sigma=(\varnothing+1+2+3)*[\sigma]$ must vanish on $\widetilde{\Gamma}$.
We set $R=(1+2+3)*[\sigma]$.
For every element $g\in\widetilde{\Gamma}$ we must have $(w*([\sigma]+R))(g)=0$ for all $w\in X^*$.
Thus in particular for all $n\geq1$ and for every element $h$ of $\stab_H(n)$ we have $(w*R)(h)=0$ for all $w\in X^{n-1}$.
This implies that $\stab_H(n)/\stab_H(n+1)$ is a subgroup of index at least $2^{3^{n-1}}$ in $\sym(X)^{X^n}$, whence the upper bound.

Now $H$ contains $\Gamma$, therefore it is enough to prove that $\stab_\Gamma(n)/\stab_\Gamma(n+1)$ has at least the desired number of elements.
Obviously the group $\Gamma/\stab_\Gamma(1)$ is isomorphic to $\sym(X)$ and thus has 6 elements.
We now consider $n\geq1$ and show that $\stab_{\Gamma'}(n)/\stab_{\Gamma'}(n+1)$ has at least $2^{2\cdot3^{n-1}}\cdot3^{3^n}$ elements.

Notice that $\Gamma'/\stab_{\Gamma'}(1)$ is isomorphic to $C_3$,
generated by the image of $[a,b]=(ab,a,b)\sigma_{123}$ for example.
Since $\Gamma'$ contains $X^n*\Gamma'$ for all $n\geq0$, this also
proves that $\stab_{\Gamma'}(n)/\stab_{\Gamma'}(n+1)$ contains a
subgroup isomorphic to $C_3^{X^n}$.  Next, we claim
$\stab_{\Gamma'}(1)/\stab_{\Gamma'}(2)$ contains a subgroup isomorphic
to a Klein four-group $V$.  Indeed, a straightforward computation
shows
\begin{align*}
[a,b][a,c]&=((ac,1,b)\sigma_{13},(1,1,1),(1,bc,a)\sigma_{23}),\\
[b,c][b,a]&=((b,1,ca)\sigma_{13},(c,ba,1)\sigma_{12},(1,1,1)),
\end{align*}
These two elements generate $V$ in the quotient
$\stab_{\Gamma'}(1)/\stab_{\Gamma'}(2)$.  Using the branching again,
this proves that $\stab_{\Gamma'}(n)/\stab_{\Gamma'}(n+1)$ contains a
subgroup isomorphic to $V^{X^{n-1}}$ for all $n\geq1$, which
terminates the proof.
\end{proof}

\begin{cor}
The closure $\gb$ of $\Gamma$ in the group $\aut X^*$ is equal to $H$.
\end{cor}
\begin{proof}
$\Gamma$ is a subgroup of $H$ and for each $n$, the quotients $H/\stab_H(n)$ and $\Gamma/\stab_\Gamma(n)$ have the same order, hence they are equal.
Since $H$ is closed in $\aut X^*$, the proposition follows.
\end{proof}

Let $\widetilde{\Gamma}_n$ be the projection of $\widetilde{\Gamma}$ to $W_n$ and let $\widetilde{M}_n$ be the kernel of the projection $\widetilde{\Gamma}_n\to\aut X^n$.
In the same way we write $\Gamma_n$ for the projection of $\Gamma$ to $W_n$, and we let $M_n$ be the kernel of the projection $\Gamma_n\to\aut X^n$.
The groups $\widetilde{M}_n$ and $M_n$ are subgroups of $A[X^n]$, and thus are elementary Abelian $2$-groups of rank at most $3^{n+1}$.
\begin{lem}\label{lem:Hanoi branch completion}
\begin{enumerate}
\item
The groups $\widetilde{M}_n$ and $M_n$ have rank $3$ when $n=0$, and rank $2\cdot3^{n-1}+2$ for all $n\geq1$.
\item
The intersection of $M_n$ with $(\Gamma_n)'$ has rank $0$ when $n=0$, and rank $2\cdot3^{n-1}$ for all $n\geq1$.
\end{enumerate}
\end{lem}
\begin{proof}
We first prove that $2\cdot3^{n-1}+2$ is an upper bound for the rank of $\widetilde{M}_n$ with $n\geq1$.
The functions in $\{w*R\mid w\in X^*,|w|\leq n-1,R\in\{Q_{1,2},Q_{1,3},Q_{2,1},Q_{2,3}\}\}$ and $\{w*\Sigma\mid w\in X^{n-1}\}$ give $2(3^n-1)+3^{n-1}$ homogeneous linear equations on $\ker(W_n\to\aut X^n)\simeq A^{X^n}$.
Moreover, it is straightforward to check that they are all linearly independent.
Thus $\widetilde{M}_n$ has rank at most $3^{n+1}-(2(3^n-1)+3^{n-1})=2\cdot3^{n-1}+2$, as we claimed.

Since $\Gamma$ is a subgroup of $\gt$, we have the inclusion $M_n\subseteq\widetilde{M}_n$.
Obviously $M_0\simeq\Gamma_0$ is isomorphic to $A$, which is of rank $3$.
Also, $\Gamma_0'$ is the trivial group.

We now consider $n\geq1$.
On the elements of the group $M_n\cap\Gamma_n'$, the functions $[a]$, $[b]$ and $[c]$ must vanish.
This gives only two more homogeneous linear equations on $M_n$, because $[a]+[b]+[c]=\Sigma$ when restricted to the group $M_n$.
This shows that $M_n\cap\Gamma_n'$ has rank at most $2\cdot3^{n-1}$.
Consider the elements
\begin{align*}
[a,b][b,c]^a&=(1,abc,bac),&
[b,c][c,a]^b&=(cba,1,bca),&
[c,a][a,b]^c&=(cab,acb,1).
\end{align*}
Their images in $M_1\cap\Gamma_1'$ are the non-trivial elements of a
Klein four-group $V$.  Since $\Gamma'$ contains $X^n*\Gamma'$ for all
$n\geq0$, we see that $M_n\cap\Gamma_n'$ contains $V^{X^{n-1}}$, which
has rank $2\cdot3^{n-1}$.

Finally, define
\begin{align*}
x_a&=b^ac=(a,a,bc),&x_b&=c^ba=(ca,b,b),&x_c&=a^cb=(c,ab,c).
\end{align*}
The images of these element in $M_1$ generate a Klein four-group $V$.
But none of these elements belong to $\Gamma'$, and therefore this group has trivial intersection with $M_1\cap\Gamma_1'$.
Moreover, we have
\begin{align*}
x_ax_bx_c\equiv x_a^2\equiv x_b^2\equiv x_c^2\equiv1\pmod{\Gamma'},
\end{align*}
and
\begin{align*}
x_a^{[c,a]}\equiv(cb,a,a)\equiv(x_a,a,a)\pmod{\Gamma'}.
\end{align*}
Thus we have the relation
\begin{align}\label{eq:Hanoi rigid kernel}
x_c\equiv x_ax_b\equiv(x_ax_b,ab,ab)\equiv(x_c,x_c,x_c)\pmod{\Gamma'},
\end{align}
and similarly $x_a\equiv(x_a,x_a,x_a)$ and $x_b\equiv(x_b,x_b,x_b)$.
We conclude that $M_n$ contains a Klein group generated by the image of $\prod_{w\in X^n}w*x_s$, with $s\in\{a,b,c\}$.
And this group has trivial intersection with $M_n\cap\Gamma_n'$.
\end{proof}

\begin{cor}
The group $\Gamma$ is dense in $\gt$.
\end{cor}
\begin{proof}
$\Gamma$ is a subgroup of $\gt$ and for each $n$, the quotients $\gt_n$ and $\Gamma_n$ have the same order, hence they are equal.
Since $\gt$ is the inverse limit of its quotients $\gt_n$, the proposition follows.
\end{proof}

\begin{proof}[Proof of Theorem~\ref{thm:BC}]
  The group $\Gamma$ embeds densely in $\gt$, which is a closed
  subgroup of $W$.  Proposition~\ref{prop:branch completion Hanoi}
  concludes.
\end{proof}

\begin{prop}
  The branch kernel $\ker(\gt\to\gb)$ is a Klein group of order 4.
\end{prop}
\begin{proof}
Reading carefully the proof of Lemma~\ref{lem:Hanoi branch completion}, one can see that Equation~\eqref{eq:Hanoi rigid kernel} and the corresponding ones for $x_a$ and $x_b$ are the key for the computation of the inverse limit
\begin{align*}
\varprojlim_{m\geq0}\left(\bigcap_{n\geq0}\stab_\Gamma(n)(X^m*\Gamma')/(X^m*\Gamma')\right)=\varprojlim_{n\geq0}\widetilde{\stab_\Gamma(n)}.
\end{align*}
Indeed, write $\bar x_a=\bar b\bar c\in A$, and similarly $\bar x_b=\bar a\bar c$ and $\bar x_c=\bar a\bar b$.
Then $\bigcap_{n\geq0}\stab_\Gamma(n)(X^m*\Gamma')/(X^m*\Gamma')$ can be seen as a subgroup of $A[X^n]$, namely the Klein group whose non-trivial elements are $X^m*\bar x_a$, $X^m*\bar x_b$ and $X^m*\bar x_c$ (with the notation $X^m*s=\sum_{w\in X^m}w*s$).
The map $X^m\to X^k$ sends $X^m*s\to X^k*s$ for all $s\in A$ and $m\geq k$, and therefore the rigid kernel of $\Gamma$ is the subgroup of $A[[X^\omega]]$ generated by $X^\omega*\bar x_a$, $X^\omega*\bar x_b$ and $X^\omega*\bar x_c$: a Klein group of order $4$.
\end{proof}

\subsubsection*{The congruence kernel of $\Gamma$}
We put together the results of the last two sections to prove just a
bit more than Theorem~\ref{thm:HCK}:

\begin{thm}\label{thm:congruence kernel Hanoi}
The congruence kernel of $\Gamma$ is an extension of a Klein group $V$ by $\Zh^3[[X^\omega]]$.
The action of $V$ is diagonal.
Each non-trivial element of $V$ acts as a half-turn along a coordinate axis on $\Zh^3$.
\end{thm}
\begin{proof}
From Equation~\eqref{eq:Hanoi rigid kernel}, it is clear that the action of $V$ is diagonal on each level.
Comparing the proof of Proposition~\ref{prop:branch kernel hanoi} and of Lemma~\ref{lem:Hanoi branch completion}, we see that it is sufficient to look at the action of $bc$ on $d$, $e$, $f$, up to cyclic permutation of $a$, $b$, $c$, and $d$, $e$, $f$.
One has
\begin{align*}
d^{bc}&\simeq d^{-1},&e^{bc}&=e,&f^{bc}&\simeq f^{-1},
\end{align*}
where the $\simeq$ sign indicates that the equality holds mod $\Gamma''$, and up to the $C_3$ factor inside $\Gamma'/\Gamma''$.
Since this $C_3$ disappears in the inverse limit defining the branch kernel, it is irrelevant here.
\end{proof}

\begin{thm}
  The congruence kernel of $\Gamma$ is torsion-free; in particular,
  the extension $\Zh^3[[X^\omega]]\cdot V$ is not split.
\end{thm}
\begin{proof}
  By Theorem~\ref{thm:congruence kernel Hanoi}, $V$ acts on the
  quotient $\Zh^3[X^0]=\Zh^3$ generated by
  $d,e,f$. We prove that the corresponding extension
  \[1\to\Zh^3\to K\to V\to 1\]
  is torsion-free; the claim follows.

  Using cyclic permutation of $a,b,c$, it is enough to prove that any
  lift of $bc\in V$ to $K$ has infinite order. Such a lift can be
  written as $x=d^\alpha e^\beta f^\gamma(bc)$ for some
  $\alpha,\beta,\gamma\in\Zh$. We compute
  \[x^2=(d^\alpha e^\beta
  f^\gamma(bc))^2=e^{2\beta}(bc)^2=e^{2\beta+1}.\] Since the equation
  $2\beta+1=0$ has no solution in $\Zh$, we see that $x^2\in\Zh^3$ is
  not trivial, and therefore $x$ has infinite order.
\end{proof}

\end{document}